\documentclass{amsart}
\newtheorem{theorem}{Theorem}
\newtheorem{proposition}{Proposition}

\newtheorem{corollary}{Corollary}
\newtheorem{lemma}{Lemma}

\newtheorem{remark}{Remark}
\newcommand{\C}{\mathbb C}

\newcommand{\Z}{\mathbb Z}
\newcommand{\R}{\mathbb R}
\newcommand{\Q}{\mathbb Q}

\begin{document}
\title[Perturbations of L-functions]
{Perturbations of L-functions with or without non-trivial zeros off
the critical line}

\thanks{First author supported by NSERC (Canada) and MEC
(Espa\~{n}a). Second author partially supported by grant
MTM2006--11391}

%% First author

\author{P. M. Gauthier}
\address{D\'{e}partement de math\'{e}matiques et de statistique
\\
         Universit\'{e} de Montr\'{e}al
\\
         CP 6128 Centre Ville
\\
         Montr\'{e}al, Qu\'{e}bec H3C 3J7
\\
         Canada}

\email{gauthier@dms.umontreal.ca}

%% Second author

\author{X. Xarles}

\address{Departament de Matem\`atiques, Edifici C, Facultat de Ci\`encies,
Universitat Aut\`onoma de Barcelona, 08193 Bellaterra (Barcelona) ,
Espa\~{n}a}

\email{xarles@mat.uab.cat}

%    General info

\subjclass[2000]{11M26, 11M41, 30E10} \keywords{$L$-functions,
Riemann Hypothesis}

\begin{abstract}
There exist small perturbations of $L$-functions, satisfying the
appropriate functional equation, for which the analogue of the
Riemann hypothesis fails radically. Moreover, this phenomenon is
generic. However, there also exist small perturbations, for which
the analogue of the Riemann hypothesis holds.
\end{abstract}
\maketitle

\section{Introduction}

In 2003, Lev D. Pustyl'nikov \cite{P} showed that the Riemann
zeta-function $\zeta(s)$ can be approximated by a function
$\zeta_\epsilon(s)$ which fails to satisfy the Riemann hypothesis.
That is, $\zeta_\epsilon(s)$ has non-trivial zeros off the critical
axis. Moreover, the function $\zeta_\epsilon(s)$ shares the
following important properties of the Riemann zeta-function. 1)
$\zeta_\epsilon(s)$ is a meromorphic function with a unique pole at
$s=1$ and assumes real values for real values of $s$. 2) The real
zeros of $\zeta_\epsilon(s)$ are the negative even integers. 3) The
function $\zeta_\epsilon(s)$ satisfies the functional equation of
the Riemann zeta-function. In 2004, Paul M. Gauthier and Eduardo S.
Zeron \cite{GZzeta} considerably improved the result of Pustyl'nikov
and in 2007, Markus Niess \cite{Ni} showed that there exist such
functions $\zeta_\epsilon(s)$ which also have interesting
universality properties.

$L$-functions are generalizations of the Riemann zeta-function and
have led to the Langlands programme and to fundamental open
questions such as the Birch and Swinnerton-Dyer conjecture.  In the
present paper, we establish results for $L$-functions analogous to
those established in \cite{GZzeta} for the Riemann zeta function.
Namely, we approximate $L$-functions by functions for which the
grand Riemann hypothesis fails.  In addition, we also establish two
results for $L$-functions which were not yet known for the Riemann
zeta function. First of all, we show that the preceding result on
approximation of $L$-functions by functions which fail to satisfy
the Riemann Hypothesis is generic. For the Riemann zeta-function,
this has been shown simultaneously and independently by Niess
\cite{Ni}. Secondly, despite this genericity of functions which fail
to satisfy the Riemann Hypothesis, we also establish the possibility
of approximating by functions for which the grand Riemann hypothesis
{\em does} hold.

Many $L$-functions $f(s)$ appearing naturally in number theory
satisfy a functional equation
$$
\Lambda_Q(f,s):=Q(s)f(s)=\Lambda_Q(f,1-s),
$$
where  $Q=Q_f$ is a meromorphic function. Examples of such functions
are the $L$-functions associated to some Dirichlet Characters, Hecke
characters, modular forms and more general automorphic forms. In
this case, the function $Q(s)$ is always of the form
$K^s\prod_{j=1}^n \Gamma(\lambda_j s + \mu _j)$, where $n$ is a
natural number, $K$ and the $\lambda_j$'s are positive real numbers
and the $\mu_j$'s are complex numbers with non-negative real part.

The Grand Riemann Hypothesis for this type of $L$-funtions asserts
that all zeros of $\Lambda_Q(f,s)$ lie on the critical axis $\Re
s=1/2$. We shall show that the conclusion of the Riemann Hypothesis
fails for most functions satisfying such a functional equation, and,
in particular, for small perturbations of such $L$-functions.

We shall say that a function $\Lambda$ is symmetric with respect to
the point $s=1/2,$ if $\Lambda(1-s)=\Lambda(s)$, for all $s\in\C.$

If $Q(s)$ is a meromorphic function on $\C$, not identically zero,
we denote by $M_Q$ the family of meromorphic functions $f$ on $\C$,
not identically zero and satisfying the functional equation
\begin{equation}\label{equation}
\Lambda_Q(f,s):=Q(s)f(s)=\Lambda_Q(f,1-s),
\end{equation}
and we denote by $H_Q$ the subfamily of entire functions in $M_Q$.
We shall say that a function $f\in M_Q$ satisfies the Riemann
hypothesis, if all zeros of $\Lambda_Q(f,s)$ lie on the critical
axis $\Re s=1/2$.

We endow $M_Q$ with the usual topology of uniform convergence on
compacta. Thus, a sequence $f_n$ of functions in $M_Q$  converges to
a function $f$ in $M_Q$, if for each compact $K\subset\C$ and each
$\epsilon>0$, there is an $n_0$ such that, for $n>n_0$, the
functions $f$ and $f_n$ have the same poles with same principal
parts on $K$ and $|f(z)-f_n(z)|<\epsilon,$ for all $z\in K$. Denote
by $RM_Q$ and $RH_Q$ the functions $f$ in $M_Q$ and $H_Q$
respectively, for which all zeros of $\Lambda_Q(f,s)$ lie on the
critical axis $\Re s=1/2$. Thus, $RM_Q$ and $RH_Q$ are respectively
the meromorphic and holomorphic  functions satisfying the functional
equation associated to $Q$ and for which the associated Riemann
hypothesis holds.

Let us say that a function $\Lambda$ defined for $s\in\C$ is
symmetric with respect to the real axis, if
$\overline{\Lambda(\overline s)}=\Lambda(s)$, for all $s\in\C.$ In
the previous papers on this topic (\cite{P},\cite{GZzeta} and
\cite{Ni}), the function approximating the Riemann zeta-function
also shared with the zeta-function the property of symmetry with
respect to the real axis. In the present paper, if the $L$-function
under consideration has this symmetry, then the approximating
functions can also be chosen to have this same symmetry. We have not
included this in the statements of our results in order to treat
more $L$-functions - not just those having this additional symmetry.
In particular, Theorem \ref{generic} and Theorem \ref{main+}, which
for the Riemann zeta-function are new, are also valid with this
"real" symmetry.

%%%%%%%%%%%%%%%%%%%%%%%%%%%%%%%%%%%%%%%%%%%%%%%%%%%%%%%%%%%%%%%%%%%%%%%%%%%%%%%%%%%%%
%%%%%%%%%%%%%%%%%%%%%%%%%%%%%%%%%%%%%%%%%%%%%%%%%%%%%%%%%%%%%%%%%%%%%%%%%%%%%%%%%%%%%%%

\section{Approximation by functions failing to satisfy the analogue of
the grand Riemann hypothesis}

If $f\in M_Q$ and $\nu\not\equiv 0$ is meromorphic on $\C$ and
symmetric with respect to the point $1/2$, then the product $\nu f$
is again in $M_Q$. It follows that, if $M_Q$ is not empty, then
there exists a function in the class $RM_Q$. That is, there exists a
function in $M_Q$ which satisfies the Grand Riemann Hypothesis.  On
the other hand the following result shows that, generically, the
Grand Riemann Hypothesis fails.

\begin{theorem}\label{generic}
In the spaces $M_Q$ and $H_Q$ of meromorphic and holomorphic
functions satisfying the functional equation (1), the classes
$M_Q\setminus RM_Q$ and $H_Q\setminus RH_Q$, for which the Riemann
Hypothesis fails are open and dense.
\end{theorem}

The following version of the Walsh Lemma on simultaneous
approximation and interpolation is due to Frank Deutsch \cite{De}.

\begin{lemma}\label{walsh}
Let $E$ be a dense subspace of a normed linear space $F$. Then, for
each $f\in F$, for each $\epsilon>0$, and for each finite choice of
continuous linear functionals $L_1,\cdots,L_m$ on $F$, there is an
$e\in E$ such that $\|e-f\|<\epsilon$ and
$$
    L_j(e)=L_j(f), \,\, j=1,\cdots,m.
$$
\end{lemma}

\begin{proof}[Proof of Theorem \ref{generic}]
Let $\{f_n\}$ be a sequence of functions in $RM_Q$ and suppose
$f_n\rightarrow f$, where $f\in M_Q$.  Suppose $f\not\in RM_Q.$
Then, $\Lambda_Q(f,\cdot)$ has a zero at a point $s_0$ not on the
critical axis. Let $d$ be less than the distance of $s_0$ from the
critical axis and also less than the distance from $s_0$ to the
nearest pole of $f$. Then, by Hurwitz's theorem, there is an $n_0$
such that for all $n>n_0$, the function $\Lambda_Q(f_n,\cdot)$ has a
zero in the disc $|s-s_0|<d.$ Thus, $f_n\not\in RM_Q$, which is a
contradiction. Thus, $RM_Q$ is closed. The proof that $RH_Q$ is
closed is similar.

To show that $RM_Q$ is nowhere dense, let $f\in M_Q$,  and consider
a basic neighborhood of $f$:
$$
    N(f,K,\epsilon)=\{g\in M_Q:\max_{s\in K}|f(s)-g(s)|<\epsilon\},
$$
where $K$ be a compact subset of $\C$ and $\epsilon$ a positive
number. We may assume that $K$ is a closed disc centered at $s=1/2.$
And that $f$ has no poles on the boundary of $K$.  Choose a point
$a$ such that the set $A=\{a,1-a,\overline a,1-\overline a\}$ is
disjoint from $K$, the critical axis, the real axis and the zeros
and poles of $\Lambda_Q(f,\cdot).$ Let $B$ be a finite set in the
interior $K^0$ of $K$ which is symmetric with respect to the point
$1/2$ and the real axis and includes all of the poles of $f$ on $K$.

Fix $\epsilon_0>0$ and let $m$ be the maximum order of the poles of
$f$ at the points of $B$. By Lemma \ref{walsh}, given
$\epsilon_1>0$, there is a polynomial $p$, such that
$|p-1|<\epsilon_1$ on $K$, for each $b\in B$, $p(b)=1, p^{(j)}(b)=0,
j=1,\cdots,m-1,$ and $p(a)=0,$ for each point $a\in A.$ We may
choose $\epsilon_1$ so small that, for the polynomial
$$\nu(s)=p(s)
\overline{p(\overline s)}
 p(1-s)\overline{p(1-\overline s)},
$$
$|\nu-1|<\epsilon_0$ on $K$.  The polynomial $\nu$ is clearly
symmetric with respect to the point $1/2$ and the real axis, and
assumes the value $0$ at the points of $A$. Moreover, we claim that
$\nu$ continues to assume the value $1$ with multiplicity at least
$m$ at each point of $B$. Since each of the four factors of $\nu$
has this property, it is sufficient to verify that, if each of two
polynomials, say $F$ and $G$, assumes the value $1$ at a point $b$
with multiplicity at least $m$, then the same is true of the product
$FG$. For each $k=1,\cdots,m-1$, the $k$-th derivative of $FG$
evaluated at $b$ is a finite sum, each of whose terms is a product,
one of whose factors is a derivative of either $F$ or $G$ of order
between $1$ and $m-1$. But these derivatives all vanish at the point
$b$. Hence, $FG$ assumes the value $1$ at the point $b$ with
multiplicity at least $m$. Thus, the polynomial $\nu$ verifies the
claim.

Set $M=\max|f|$ on $\partial K.$ Now, choose $\epsilon_0=\epsilon/M$
and set $g=\nu f.$ Then, $g$ satisfies the functional equation and
$\Lambda_Q(g,\cdot)$ has zeros at the points of $A$.  On $\partial
K$  we have the estimate $|f-g|<\epsilon.$  Since, $1-\nu$ has a
zero of order at least $m$ at each pole of $f$ in $K^o$, it follows
that $f-g=(1-\nu)f$ is holomorphic in $K^o$. By the maximum
principle, the estimate $|f-g|<\epsilon$ holds on all of $K$.

Thus, $g$ is in the neighborhood $N(f,K,\epsilon)$ of $f$ and not in
$RM_Q.$ Since $f$ was arbitrary in $M_Q$, the complement of $RM_Q$
in $M_Q$ is open.  Thus the closed set $RM_Q$ is nowhere dense in
$M_Q$.

If $f$ were initially chosen in $H_Q$, then the corresponding $g$
would be in $RH_Q$ and so the same proof shows that also that the
closed set $RH_Q$ is nowhere dense in $H_Q.$

\end{proof}

In the proof of Theorem \ref{generic}, we invoked a Walsh type lemma
in order to simultaneously approximate on a compact set and
interpolate at finitely many points. In the sequel, we shall need to
approximate on a (possibly unbounded) closed set and simultaneously
interpolate on a (possibly infinite) discrete subset thereof.

For a closed subset $X\subset\C$, we denote by $A(X)$ the space of
functions continuous on $X$ and holomorphic on $X^0.$ Every function
$f:X\rightarrow\C$, which can be uniformly approximated on $X$ by
entire functions, is necessarily in the class $A(X).$  A closed set
$X$ is said to be a set of {\em uniform} approximation if, for every
$f\in A(X)$ and every positive {\em constant} $\epsilon$, there is
an entire function $g$ such that $|f(z)-g(z)|<\epsilon,$ for all
$z\in X.$ A theorem of Norair U. Arakelian (see \cite{Gai}) asserts
that a closed subset $E\subset\C$ is a set of uniform approximation
if and only if $\overline\C\setminus E$ is connected and locally
connected, where $\overline\C$ denotes the closed complex plane
$\C\cup\{\infty\}.$

A closed set $X$ is said to be a set of {\em tangential}
approximation if, for every $f\in A(X)$ and every positive
continuous {\em function} $\epsilon$ on $X$, there is an entire
function $g$ such that $|f(z)-g(z)|<\epsilon(z),$ for all $z\in X.$
By the Tietze extension theorem for closed sets \cite{Du}, it makes
no difference whether we take the function $\epsilon$ to be defined
on $X$ or on $\C$. Of course, a set of tangential approximation is
necessarily a set of uniform approximation. Let us say that a family
$\{E_j\}$ of subsets of $\C$ has {\em no long islands}, if, for each
$r>0$, there is a (larger) $r'$ such that no $E_j$ meets both
circles $|z|=r$ and $|z|=r'$. This condition was introduced by the
first author \cite{G}, who showed that, for a closed set $E$ of
uniform approximation to be a set of tangential approximation, a
necessary condition is that the family of components of the interior
$E^0$ have no long islands. Ashot H. Nersessian \cite{N} showed that
this condition is also sufficient. For an overview of uniform and
tangential approximation, see the book of Dieter Gaier \cite{Gai}.

An example of a set of uniform approximation, would be the set $E_0$
consisting of the union of the critical strip $0\le\Re s\le1$ and
the real axis $\Im s=0.$ However, this is not a set of tangential
approximation, since the interior has an unbounded component. Let
$\{a_j\}$ and $\{b_j\}$ be sequences of positive numbers strictly
increasing to infinity, with $a_j<b_j<a_{j+1}$, for each $j$. For
each $j$, let $A_j$ be the closed {\em double rectangle}
$A_j=\{z:a_j\le|\Im z|\le b_j, |\Re z|\le j\}$. Let $E_1$ be the
union of the real axis, the critical axis and all of the double
rectangles:
$$
    E_1=\{z:\Im z=0\}\cup\{z:\Re z=1/2\}\cup\bigcup_jA_j.
$$
The set $E_1$ is an example of a set of tangential approximation.

The following lemma was given in \cite{GZzeta}, where it was claimed
that the proof followed from another paper. For completeness, we
provide the proof.

\begin{lemma}\label{walsh tangential finite}
Let $X$ be a closed set of tangential approximation in $\C$. Let $B$
be a finite subset of $X$ and for each $b\in B,$ let $m_b$ be a
natural number, with the restriction that $m_b=1$ if $b\in\partial
X.$ Then, for every $f\in A(X)$ and for every positive continuous
function $\epsilon: X\rightarrow\R,$ there exists an entire function
$F$, such that the following simultaneous approximation and
interpolation holds: $|f(z)-F(z)|<\epsilon(z)$ for every $z\in X$,
and $f-F$ has a zero of multiplicity (at least) $m_b$, for every
$b\in B.$
\end{lemma}

\begin{proof}
For $g:X\rightarrow \C$, set
$$
    \|g\|_\epsilon = \sup_{z\in X}\frac{|g(z)|}{\epsilon(z)}
$$
and
$$
    A_\epsilon(X)=\{g\in A(X):\|g\|_\epsilon <\infty\}.
$$
Then, $A_\epsilon(X)$ is a normed linear space. Moreover, for each
$z\in X$, the mapping $g\longmapsto g^{(j)}(z)$ is a continuous
linear functional on $A_\epsilon(X)$, for all $j=0,1,2,\cdots,$ if
$z\in X^o$, and for $j=0$, if $z\in\partial X.$

Now, let $f\in A(X)$. Since $X$ is a set of tangential
approximation, there exists an entire function $h$ such that
$|h-f|<\epsilon$ on $X$.  Thus, $g=f-h$ is in $A_\epsilon(X)$.
Since, $X$ is a set of tangential approximation, the entire
functions are dense in $A_\epsilon(X)$. By Lemma \ref{walsh} there
is an entire function $G$ such that $|G-g|<\epsilon$ on $X$ and
$G^{(j)}(b)=g^{(j)}(b)$, for $b\in B$ and $j=0,\cdots,m_b-1.$  Set
$F=G+h.$
\end{proof}

We shall make use of the above lemma and an induction process to
approximate while simultaneously interpolating on an {\em infinite}
set $B$. First, we introduce a type of set on  which this is
possible.

Let us say that a closed set $X\subset\C$ is a {\em chaplet} if
there is a strictly increasing sequence $\{r_n\}$ of positive
numbers tending to $\infty$, such that the exhaustion of $\C$ by the
discs $K_n=\{z:|z|\le r_n\}$ has the following compatibility
conditions with the set $X$.

{\em Condition (i).} Setting $K_0=\emptyset$, each set $Y_n=X\cup
K_{n-1}$ is a set of tangential approximation.

{\em Condition (ii).} $X\cap\partial K_n$ is finite, for each $n$.

\begin{lemma}\label{walsh tangential infinite}
Let $X$ be a chaplet in $\C$. Let $B$ be a discrete subset of $X$
and for each $b\in B,$ let $m_b$ be a natural number, with the
restriction that $m_b=1$ if $b\in\partial X.$

Then, for every $f\in A(X)$ and for every positive continuous
function $\epsilon: X\rightarrow\R,$ there exists an entire function
$h$, such that the following simultaneous approximation and
interpolation holds: $|f(z)-h(z)|<\epsilon(z)$ for every $z\in X$,
and $f-h$ has a zero of multiplicity (at least) $m_b$, for every
$b\in B.$
\end{lemma}

\begin{proof}A similar result on simultaneous approximation and
interpolation was obtained in \cite{GH}, but for uniform
approximation. Here, we need the stronger tangential approximation
and so we shall imitate the proof in \cite{GH}, but with appropriate
modifications to obtain tangential approximation. We also adapt a
technique from \cite{GZdendrite}

We may assume that $\epsilon$ is defined on all of $\C$ and, in
fact, is a strictly decreasing function of $|s|$, for $s\in\C.$ Let
$X,B,f$ and $m_b$ be as in the hypotheses of the lemma and $\{K_n\}$
be an exhaustion of $\C$ compatible with the chaplet $X$. We may
assume that $B$ includes the finite sets $X\cap\partial K_n$, for
each $n$.

Since, $X$ is a set of tangential approximation, the entire
functions are dense in $A_\epsilon(X)$. Set $F_0=f$. It follows from
Lemma \ref{walsh tangential finite} that there is an entire function
$F_1$ such that

$$
    |F_1(z)-f(z)|<\epsilon(z)\cdot 2^{-1},  \,\,\,\, z\in X,
$$
$$
    |F_1(z)-F_0(z)|<\epsilon(z)\cdot 2^0,  \,\,\,\, z\in K_0,
$$
$$
    F_1^{(j)}(b)=f^{(j)}(b), \,\,\,\, b\in B\cap K_1,  \,\,
    j=0,\cdots,m_b-1,
$$

Note that, the second equation is vacuous. We proceed by induction.
Given an entire function $F_k$ such that
\begin{equation}\label{induc1}
    |F_k(z)-f(z)|<\epsilon(z)\cdot\sum_{j=1}^k2^{-j},  \,\,\,\, z\in X,
\end{equation}
\begin{equation}\label{induc2}
    |F_k(z)-F_{k-1}(z)|<\epsilon(z)\cdot 2^{-k},  \,\,\,\, z\in K_{k-1},
\end{equation}
\begin{equation}\label{induc3}
    F_k^{(j)}(b)=f^{(j)}(b), \,\,\,\, b\in B\cap K_k,  \,\, j=0,\cdots,m_b-1,
\end{equation}
we define an associated function $h_k\in A(X\cup K_k)$ as follows:
$$
h_k(z)= \left\{
\begin{array}{ll}
    F_k(z)  &   z\in K_k,\\
               &       \\
    f(z)  &   z\in X\setminus K_k.
\end{array}
\right.
$$
The induction step is then: By Lemma \ref{walsh tangential finite},
there exists an entire function $F_{k+1}$ such that
$$
    |F_{k+1}(z)-h_k(z)| < \epsilon(z)\cdot 2^{-(k+1)}, \,\,\,\, z\in X\cup K_k
$$
and
$$
    F_k^{(j)}(b)=h_k^{(j)}(b), \,\,\,\, b\in B\cap K_{k+1},  \,\,
    j=0,\cdots,m_b-1.
$$
Then, (\ref{induc1}), (\ref{induc2}) and (\ref{induc3}) hold with
$k$ replaced by $k+1,$ and hence, by induction, for all
$k=1,2,\cdots.$

By (\ref{induc2}), the sequence $\{F_k\}$ converges to an entire
function $F$.  The function $F$ performs the required approximation
by (\ref{induc1}), and  performs the required interpolation by
(\ref{induc3}).
\end{proof}

In the following proposition, we shall apply Lemma \ref{walsh
tangential infinite} in order to approximate and interpolate the
constant function $1$, while imposing zeros on the approximating
function. A version of this proposition was presented in
\cite{GZzeta}, but therein the interpolation of $1$ was only at
finitely many points and the interpolation was not with
multiplicity.

\begin{proposition}\label{chaplet}
Let $X$ be a chaplet in $\C$. Let $B$ be a discrete subset of $X$
and for each $b\in B,$ let $m_b$ be a natural number, with the
restriction that $m_b=1$ if $b\in\partial X.$  Let $A$ be a discrete
set which is disjoint from $X$. Suppose $X,B,$ and $A$ are all
symmetric with respect to the point $1/2$ and the real axis, and
$m_b$ also, that is, $m_b=m_{\overline b}=m_{1-b}.$ Then, for every
strictly positive continuous function $\epsilon: X\rightarrow\R,$
there exists a function $\nu$ holomorphic on $\C$, symmetric with
respect to the point $1/2$ and the real axis, and whose zeros are
precisely the points of $A.$ Moreover, $\nu$ simultaneously
approximates and interpolates $1,$ in that $|\nu(z)-1|<\epsilon(z),$
for each $z\in X$, and $\nu-1$ has a zero of multiplicity (at least)
$m_b$, at each point $b\in B.$
\end{proposition}

\begin{proof}
We merely sketch the proof, since it is similar to that of
Proposition 4 in \cite{GZzeta}. We may assume that $\epsilon$ is
symmetric with respect to the point $1/2$ and the real axis. Let $g$
be an entire function, whose zero set is precisely $A$ and let $G$
be a branch of $\log g$ on $X$. By Lemma \ref{walsh tangential
infinite}, there is an entire function $h$ such that
$|h-G|<\epsilon/12$ on $X$ and $h-G$ has a zero of multiplicity (at
least) $m_b$, at each point $b\in B.$ The function
$$
    \nu(s)=\frac
    {g(s)\overline{g(\overline s)}g(1-s)\overline{g(1-\overline s)}}
    {exp\left(h(s)+\overline{h(\overline s)}+h(1-s)+\overline{h(1-\overline s)}\right)}
$$
has all of the required properties. The additional property that we
must verify, that was not proved in \cite{GZzeta} is that $\nu$ not
only interpolates the constant $1$ at the points $b\in B$, but in
fact interpolates up to multiplicity $m_b$.  That is, we must verify
that $\nu^{(j)}(b)=0$, for $j=1,\cdots,m_b-1.$  The proof of this is
the same as in the proof of Theorem \ref{generic}, since each of the
four factors of $\nu$ has this property.
\end{proof}

The next theorem asserts that we may approximate a function $f\in
M_Q$ extremely well by one which satisfies the same functional
equation and fails radically to satisfy the analogue of the grand
Riemann hypothesis.

\begin{theorem}\label{main-}
Let $f$ be a function in $M_Q$. Let $A$ and $B$ be disjoint discrete
sets, symmetric with respect to the point $1/2$ and the real axis
and suppose $B$ contains the zeros and poles of $f$. For each $b\in
B,$ let $m_b$ be a natural number, also symmetric with respect to
the point $1/2$ and the real axis; that is, $m_b=m_{\overline
b}=m_{1-b}.$ Let $\epsilon$ and $\eta$ be strictly positive
continuous functions on $\C$ and $[0,+\infty)$ respectively.  Then,
there exists a function $g\in M_Q$, $g\not=f$, and a closed set
$X\subset\C$, containing the real and critical axes, such that:

(i) $\Lambda_Q(g,a)=0$, for $a\in A$ while $\Lambda_Q(g,\cdot)$ and
$\Lambda_Q(f,\cdot)$ have exactly the same zeros with same
multiplicities on $\C\setminus A$;

(ii) $f$ and $g$ have the same poles and at each point $b\in B$, the
first $m_b$ terms of the Laurent series of $f$ and $g$ coincide;

(iii) $|g(z)-f(z)|<\epsilon(z)$, for $z\in X$;

(iv) area$\{z:z\not\in X, |z|>r\}<\eta(r)$, for $r\in[0,+\infty).$
\end{theorem}

\begin{proof} We begin by constructing a chaplet $X$. Let $\{\alpha_j\}$ and
$\{\beta_j\}$ be sequences of positive numbers strictly increasing
to infinity, such that $\alpha_j<\beta_j<\alpha_{j+1}$. Denote by
$\mathcal{A}_j$ the annulus
$$
    \mathcal{A}_j = \{z:\alpha_j<|z-1/2|<\beta_j\},
$$
and by  $\mathcal A$ the union of these annuli. We may choose the
sequences $\{\alpha_j\}$ and $\{\beta_j\}$ so that $\mathcal{A}_j$
contains $A$, that is, such that for all $a\in A,$ there is a $j$
with $a\in\mathcal{A}_j.$

Choose a point $1/2+p$, with $p=\alpha_1e^{i\theta},
0<\theta<\pi/2,$ lying on the circle $(|z-1/2|=\alpha_1)$, such that
the ray $R=\{z=1/2+tp:1\le t<+\infty\}$ is disjoint from the set
$B$. Let
$$
    S_p = \cup_{t\ge 1}\{z:|z-(1/2+tp)|<\delta(t)\}
$$
be a neighborhood of the ray $R$, where $\delta$ is a positive
continuous function which decreases so rapidly that the strip $S_p$
is disjoint from the critical and real axes and from the set $B$.
Denote by $\tilde S_p$ the reflection of $S_p$ with respect to the
real axis and denote by $1-S_p$ and $1-\tilde S_p$ respectively the
reflection of $S_p$ and $\tilde S_p$ with respect to the point
$z=1/2$.  Denote by $S$ the union of these four strips. For each
$b\in B,$ let $Q_b$ be a closed disc centered at $b$, such that the
radii are symmetric with respect to the real axis and the point
$1/2.$ Set $\mathcal{B}=\cup_{b\in B}Q_b$.  Finally, set
$$
   X = (\C\setminus(\mathcal A\cup S))\cup\mathcal{B}\cup\{z:\Im z=0\}\cup\{z:\Re
    z=1/2\}.
$$
Then, if the $Q_b$ are sufficiently small, $X$ is a chaplet and
$X,B, m_b$ and $A$ satisfy the hypotheses of Proposition
\ref{chaplet}. Moreover, we may construct $X$ so that the complement
is as small as we please. In fact, we may construct $X$ so that
condition (iv) is satisfied.

For each pole $b$, let $K_b$ be a closed disc centered at $b$ and
contained in $X^0$ such that the discs $K_b$ are disjoint and form a
locally finite family.  Now, let $\epsilon_1(z)$ be a strictly
positive continuous function on $X$ such that, for every $z\in X,$
$$
    \epsilon_1(z)<\left\{
\begin{array}{ll}
\epsilon(z)/|f(z)|,  &    z\in X\setminus\cup_pK_b;\\
(\min_{K_b}\epsilon)/(\max_{\partial K_b}|f|),           & z\in K_b.
\end{array}
    \right.
$$
Let $\nu$ be an entire function corresponding to
$\epsilon_1,A,B,m_b$ in Proposition \ref{chaplet}. If $b$ is a pole
of $f$, we may and shall assume that $m_b$ is the multiplicity of
the pole.  Set $g=\nu f.$ Since $\nu$ is entire and assumes the
value $1$ at each pole of $f$, and with  the same multiplicity as
that of the pole, the function $g$ has the same poles as $f$, and
with the same principal parts. Thus, $f-g$ is an entire function.
Now, $|f(z)-g(z)|<\epsilon(z)$, for $z\in X\setminus\cup_pK_b.$ For
$z\in K_b$, we have, by the maximum principle,
$$
    |f(z)-g(z)| \le \max_{\partial K_b}|f(z)-g(z)| =
    \max_{\partial K_b}(|1-\nu(z)|\cdot|f(z)|) < \epsilon(z).
$$
We have shown that $|f(z)-g(z)|<\epsilon(z)$, for all $z\in X.$ The
function $g$ satisfies conditions (i)-(iv).

In case $f$ is the Riemann zeta-function and the Riemann hypothesis
is false, then, of course, $f$ is itself a (perfect) approximation
of itself failing to satisfy the Riemann hypothesis.  However, in
the present theorem, it is easy to assure that $g\not= f$, by simply
choosing the set $A$ such as to contain a point not among the zeros
of $f$. Thus, the theorem is non-trivial, even in case the Riemann
hypothesis fails.
\end{proof}

By a continuous perturbation of a function $f$ in $M_Q$, we
understand a continuous curve in $M_Q$:
$$
    [0,1)\rightarrow M_Q
$$
$$
    t\mapsto f_t,
$$
such that $f_0=f.$

\begin{corollary} If $f$ is an $L$-function, satisfying a
functional equation (\ref{equation}), then there is a continuous
perturbation, $f_t, 0\le t<1, f_0=f$, such that each $f_t$, for
$0<t<1$, satisfies the same functional equation but fails to satisfy
the analogue of the Riemann hypothesis.
\end{corollary}

\begin{proof}
Let $M_Q$ be the class of meromorphic functions associated to the
functional equation for $f$. Let $B$ be a discrete set, symmetric
with respect to the point $1/2$ and the real axis and containing the
zeros and poles of $f$. Let $\{a_n\}$, with $n<|a_n|<n+1$, be a
sequence distinct from the zeros of $f$, and disjoint from the real
axis, the critical axis and $B$. Set $A_n=\{a_k:k>n\}$ and
$\epsilon_n=1/n$. From Theorem \ref{main-}, we obtain a function
$g_n\in M_Q$ and a closed set $X_n$ such that
$$
    |f(z)-g_n(z)|<1/n, \,\,\,\, z\in X_n
$$
and
$$
    g_n(a_k)=0, \,\,\,\, k>n.
$$
As in the proof of Theorem \ref{main-}, we see that $X_n$ may be so
chosen that it contains the closed disc, centered at the origin and
of radius $n.$

For each $n$ and for $n\le t\le n+1$, set
$$
    g_t = (1-t+n)g_n + (t-n)g_{n+1}.
$$
Then, $g_t\in M_Q$,
$$
    |f(z)-g_t(z)|<1/n, \,\,\,\, |z|\le n, \,\,\,\, n\le t\le n+1,
$$
and
$$
    g_t(a_k)=0, \,\,\,\, k>n+1.
$$
Thus, $t:[1,+\infty)\rightarrow g_t$ is a continuous path of
meromorphic functions in $M_Q$, each of which has infinitely many
zeros different from the zeros of $f$ and such that $g_t\rightarrow
f$.

To conclude the proof, we have only to reparametrize by setting
$f_0=f$ and  $f_t=g_{1/t}, 0<t<1.$
\end{proof}

%%%%%%%%%%%%%%%%%%%%%%%%%%%%%%%%%%%%%%%%%%%%%%%%%%%%%%%%%%%%%%%%%%%%%%%%%%%%%%%%%%%%%
%%%%%%%%%%%%%%%%%%%%%%%%%%%%%%%%%%%%%%%%%%%%%%%%%%%%%%%%%%%%%%%%%%%%%%%%%%%%%%%%%%%%%%%

\section{Approximation by functions satisfying the analogue of
the grand Riemann hypothesis}

Having approximated zeta-functions by functions which do {\em not}
satisfy the analogue of the Riemann hypothesis, we now consider the
possibility of approximating by functions which {\em do} satisfy the
analogue of the Riemann hypothesis.  Of course, we cannot have our
cake and eat it too. If we {\em could} approximate the Riemann
zeta-function uniformly on compacta by functions which satisfy the
analogue of the Riemann hypothesis, this would prove the Riemann
hypothesis. Although we do not know how to approximate
zeta-functions {\em uniformly on compacta} by functions satisfying
the analogue of the Riemann hypothesis, the next theorem shows that
we can, nevertheless, approximate the Riemann zeta-function {\em
extremely well} by functions all of whose non-trivial zeros lie on
the critical axis and hence satisfy the analogue of the Riemann
hypothesis.

First we need a lemma on approximating zero-free functions by entire
zero-free functions.

\begin{lemma}\label{zero-free}
Let $X$ be a  chaplet in $\C$. Let $f\in A(X)$ and let $\epsilon$ be
a strictly positive continuous function on $X$. Suppose $f$ is
zero-free on $X$. Then, there is a zero-free entire function $e^h$
such that
$$
    |e^{h(z)}-f(z)|<\epsilon(z),\,\,\,\, \mbox{for } z\in X.
$$
Moreover, if $B$ is a discrete subset of $\C$ lying in the interior
of $X$ and for each $b\in B$, $m_b$ is a positive integer, we may
find such an $h$ so that $e^h$ interpolates $f$ to order $m_b$ at
each point $b\in B$. Moreover, if $X$, $f$ and $B$ are symmetric
with respect to $1/2$, then we may take $h$ symmetric with respect
to $1/2$.
\end{lemma}

\begin{proof} For the function $f\in A(X)$, a function $F\in A(X)$ is needed such that $e^F =
f.$ In the article \cite{GP}, it is stated in Lemma 2 that this is
the case if $X$ is a set of uniform approximation. Since $X$ is a
chaplet it is indeed a set of uniform approximation. Moreover, it is
easy to see that $F\in A(X)$.

Since the exponential function is uniformly continuous on compacta,
it follows that for any positive continuous function $\eta$, there
is a positive continuous function $\delta$ such that, if $w$ is any
continuous function on $X$ such that $|w|<\delta$, then
$e^{|w|}-1<\eta$. Now set $\eta=\epsilon/|e^F|$. Since $X$ is a
chaplet, it follows from Lemma \ref{walsh tangential infinite} there
is an entire function $h$ such that $|h(z)-F(z)|<\delta(z)$ for
$z\in X$.
$$
|e^h-f|=|e^h-e^F|=|e^{h-F}-1||e^F|\le(e^{|h-F|}-1)|e^F|<\epsilon.
$$
Moreover, we may choose $h$ so that $e^h$ interpolates $f$ at the
points of $B$. Indeed, we may choose $h$ so as to interpolate $F$ to
order $m_b$ at each point $b$.  Thus, $h-F$ has a zero of order
$m_b$ at $b$. Since the exponential function is locally
biholomorphic, the function $e^{h-F}$ assumes the value $1$ with
multiplicity $m_b$ at $b$. Since, $e^F$ is different from zero at
the point $b$,  the product $(e^{h-F}-1)e^F$ also has a zero of
order $m_b$ at $b$. That is, $e^h-f = (e^{h-F}-1)e^F$ has a zero of
order $m_b$ at $b$. So $e^h$ interpolates $f$ to order $m_b$ at the
point $m$.

Suppose, $X$, $f$ and $B$ are symmetric with respect to $1/2$. Since
$f(z) = f(1-z)$ and $e^F = f,$ we have that
$$
    g(z) := \frac{F(z)-F(1-z)}{2\pi i}\in \Z, \,\,\,\,\,\,\,\,\mbox{for all}\,\, z\in X.
$$
The function $g$ is continuous on $X,$ therefore constant on each
component of $X$. The components of $X$ come in pairs. Each member
of a pair is the reflection of the other with respect to $1/2$. Let
$X_c$ be a collection of components, one from each pair, indexed by
a parameter $c$. We may denote the component paired with $X_c$ by
$1-X_c$. The function $g$ is constant on $X_c$ and so there is an
integer $k_c$ such that $F(z)=F(1-z)+k_c2\pi i$, for $z\in X_c$.
Note that if $X_c$ is the same as its paired component, then on
$X_c$, we have $F(z)=F(1-z)+k_c2\pi i=F(z)+k_c2\pi i+k_c2\pi i$ and
so $k_c=0$. Now, we had some freedom in choosing the function $F$.
Let us modify $F$ by leaving $F$ as it is on $1-X_c$ but replacing
it on $X_c$ by $F-k2\pi i.$ With this new definition of $F$, we have
that $g\equiv 0$ on each $X_c.$ Thus, $F$ is symmetric with respect
to $1/2$. Now, we may assume that $h$ is also symmetric with respect
to $1/2$ by replacing $h(z)$ by $(h(z)+h(1-z))/2$. This proves the
lemma.
\end{proof}

In Theorem \ref{main-} we approximated a function $f\in M_Q$ by a
function which satisfies the same functional equation and fails to
satisfy the grand Riemann hypothesis. The following theorem asserts
that we may also approximate by one which {\em does} satisfy the
analogue of the grand Riemann hypothesis.

\begin{theorem}\label{main+}
Let $f$ be a function in $M_Q$ and suppose (as for the Riemann
zeta-function) that no pole of $f$ is a zero of
$\Lambda_Q(f,\cdot)$. Let $\epsilon$ and $\eta$ be strictly positive
continuous functions on $\C$ and $[0,+\infty)$ respectively.  Then,
there exists a function $g\in M_Q$, $g\not=f$, and a closed set
$X\subset\C$, containing the real and critical axes, such that:

(i) on the critical axis, $\Lambda_Q(g,\cdot)$ has the same zeros
with the same multiplicities as $\Lambda_Q(f,\cdot)$;

(ii) $\Lambda_Q(g,\cdot)$ has no zeros off the critical axis;

(iii) $\Lambda_Q(g,\cdot)$ has the same poles with the same
multiplicities as $\Lambda_Q(f,\cdot)$;

(iv) $|g(z)-f(z)|<\epsilon(z)$, for $z\in X$;

(v) area$\{z:z\not\in X, |z|>r\}<\eta(r)$, for $r\in[0,+\infty).$
\end{theorem}

\begin{proof}
As in the proof of Theorem \ref{main-}, we construct a chaplet $X$,
symmetric with respect to $1/2$ and the real axis, containing the
real and critical axes and satisfying condition (v). Moreover, by
slightly deforming the annuli in the construction of $X$, we may so
construct $X$ that it excludes all of the zeros of
$\Lambda_Q(f,\cdot)$ off the critical axis, while retaining the
condition that $X$ contains all of the poles of $f$ in its interior.

Let $f_o$ be an entire function, symmetric with respect to $1/2$,
whose zeros are precisely the zeros of $\Lambda_Q(f,\cdot)$ off the
critical axis and with the same multiplicities.

Let $B$ be the set of poles of $f$ and their reflections with
respect to the point $1/2$. As in the proof of Theorem \ref{main-},
for each $b\in B$, let $K_b$ be a closed disc centered at $b$ such
that: $K_b$ is contained in $X^0$; the function $f$ has no pole in
$K_b$ other than possibly $b$ ; the discs $K_b$ are disjoint and
form a locally finite family. Let $M_b$ be the minimum of
$\epsilon(z)$, for $z\in K_b$  and set
$$
    N_b = \frac{\max_{z\in\partial K_b}|f(z)|}{\min_{z\in
    K_b}|f_o(z)|}.
$$
Note that $f_o$ has no zeros on $X$ and in particular on $K_b.$ Now,
let $\epsilon_1(z)$ be a strictly positive continuous function on
$X$ such that, for every $z\in X,$
$$
    \epsilon_1(z)<\left\{
\begin{array}{ll}
\epsilon(z)|f_o(z)|/|f(z)|,  &    z\in X\setminus\cup_bK_b;\\
M_b/N_b,                 &    z\in K_b.
\end{array}
    \right.
$$

By Lemma \ref{zero-free}, there is a zero-free entire function $e^h$
symmetric with respect to $1/2$ such that
$$
    |e^h-f_o|< \epsilon_1 \,\, \mbox{on } X
$$
and, if $b$ is any pole of $f$ and the order of the pole is $m_b$,
then $e^h$ interpolates $f_o$ to order $m_b$ . Hence, $e^hf-f_of$ is
in $A(X)$ and so by the maximum principle, as in the proof of
Theorem \ref{main-},
$$
    |e^hf-f_of|<\epsilon|f_o| \,\, \mbox{on } X.
$$
Thus, if we put $g=(e^h/f_o)f$, then
$$
    |g-f|<\epsilon \,\, \mbox{on } X.
$$
As we remarked earlier, because $g$ is the product of a function
$e^h/f_o$ symmetric with respect to $1/2$ and a function $f$ in
$M_Q$, it follows that $g$ is also in $M_Q$. That is, $g$ satisfies
the functional equation. Moreover, $\Lambda_Q(g,\cdot)$ has the same
zeros as $\Lambda_Q(f,\cdot)$ with the same multiplicities on the
critical axis and no other zeros and $\Lambda_Q(g,\cdot)$ has the
same poles as $\Lambda_Q(f,\cdot)$ with the same multiplicities.

We have approximated $f$ by a function $g$ which satisfies the
analogue of the grand Riemann hypothesis. In case $f$ is the Riemann
zeta-function and the Riemann hypothesis holds, the Riemann
zeta-function is, of course, a (perfect) approximation of itself
satisfying the Riemann hypothesis. However, it is easy in our
theorem to make sure that $g\not=f$. For example, the function $h$
is highly non-unique and so not all of the corresponding functions
$g=(e^h/f_o)f$ can be $f$.  Thus, even in case the Riemann
hypothesis holds, Theorem \ref{main+} is non-trivial for the Riemann
zeta-function.
\end{proof}

As a corollary (of the proof of the preceding theorem) we have the
following product decomposition for the Riemann zeta-function.

\begin{corollary}\label{product}
The Riemann zeta-function can be written as a product
$\zeta=b\zeta_+$, where $b$ has all of the non-trivial zeros of
$\zeta$ off the critical axis and $\zeta_+$ is an "ideal" zeta
function in the sense that it has no non-trivial zeros off the
critical axis, has a single simple pole at $z=1$, satisfies the
functional equation and approximates the Riemann zeta-function
extremely well in the sense of the previous theorem.
\end{corollary}

\begin{proof}
In the proof of the previous theorem, we have $g=(e^h/f_o)f$. Take
$f=\zeta$ and set $b=f_oe^{-h}$ and $\zeta_+=g$.
\end{proof}

The Riemann hypothesis would of course follow, if one could show
that $b=1$ in this decompostion. But the product decomposition in
Corollary \ref{product} is not unique, since the approximating
function $g$ in Theorem \ref{main+} is not unique. An improved
version of Theorem \ref{main+} is therefore desirable, in which the
class of approximating functions $g$ is reduced to functions
satisfying even more properties of the Riemann zeta function.  In
Section 4 we shall discuss the possibility of representing the
approximating functions in Theorem \ref{main-} and Theorem
\ref{main+} by Dirichlet series.

{\bf Philosophical remarks.} The above approximations are so strong
that the approximator "cannot be distinguished" from the
approximatee. Suppose a mischievous  angel (a devil?) were to hand
us functions $\zeta_+$ and $\zeta_-$, which approximate the Riemann
zeta-function as in our theorems and satisfy the corresponding
functional equation, where $\zeta_+$ satisfies the analogue of the
Riemann hypothesis and $\zeta_-$ does not . If the angel claimed
that these functions were in fact the Riemann zeta-function, there
is no way we could prove it wrong. That is, we could not distinguish
these functions from each other or from the Riemann zeta-function.
Indeed, we can choose the speed of approximation $\epsilon$ so small
that on $X$ these functions differ by each other by less than the
diameter of an electron and so no (present or future) scientific
instrument could distinguish their values on $X$. While it is true
that on $\C\setminus X$, these functions may differ greatly, this
set $\C\setminus X$ is so small, that if the fastest conceivable
computer (or finite team of computers) where to pick points of $\C$
successively at random, it is likely that the sun would grow cold
before the computer would fall upon a point of $\C\setminus X$. We
thus have functions $\zeta_+$ and $\zeta_-$ satisfying the
functional equation, "indistinguishable" from the Riemann
zeta-function and from each other, such that $\zeta_+$ has no
non-trivial zeros off the critical axis while $\zeta_-$ does.

%%%%%%%%%%%%%%%%%%%%%%%%%%%%%%%%%%%%%%%%%%%%%%%%%%%%%%%%%%%%%%%%%%%%%%%%%%%%%%%%%%%%%%%
%%%%%%%%%%%%%%%%%%%%%%%%%%%%%%%%%%%%%%%%%%%%%%%%%%%%%%%%%%%%%%%%%%%%%%%%%%%%%%%%%%%%%%%

\section{Dirichlet series}

Let $\zeta_+$ and $\zeta_-$ be functions as in the above theorems,
which approximate the Riemann zeta function, are practically
indistinguishable and respectively satisfy and do not satisfy the
analogue of the Riemann hypothesis.

We recall that both functions $\zeta_+$ and $\zeta_-$ approximate
the Riemann zeta function $\zeta$ exceedingly well; the function
$\zeta_+$ has no non-trivial zeros off the critical axis, whereas
the function $\zeta_-$ can be chosen to have a multitude of
non-trivial zeros off the critical axis. We should like to say
something concerning the representations of these functions by
Dirichlet series. To this end we shall employ a result of
Fr\'ed\'eric Bayart \cite{B} establishing  the existence of a
universal Dirichlet series.

{\bf Definition.} A compact subset $K$ of $\C$ is said to be {\it
admissible} if ${\C}\setminus K$ is connected, and if $K$ can be
written $K=K_1\cup\cdots\cup K_d,$ each $K_i$ being contained in a
strip $S_i=\{z:a_i\le\Re(z)\le b_i\},$ with $b_i-a_i<1/2,$ the
strips $S_i$ being disjoint.

\begin{theorem}\cite[Remark 5]{B}\label{bayart} There exists a Dirichlet series
$S(z)=\sum_{n\ge 1}a_nn^{-z}$  absolutely convergent in the open
right half-plane $\{ z:\Re(z)>1\}$, which is universal in the
following sense. For each admissible compact set $K$ in the closed
left half-plane $\{ z:\Re(z)\le 1\}$ and each function $g$
continuous on $K$ and holomorphic on the interior of $K$, there is a
sequence of partial sums of $S$ which converges to $g$ uniformly.
The set of such Dirichlet series is dense in the space of Dirichlet
series absolutely convergent in $\{z:\Re(z)>1\}.$
\end{theorem}

Bayart states his theorem for the right half-plane $\Re s>0$, but
the substitution $s=z-1$ yields the above version. From this, we can
obtain the following. For a real number $\sigma$, denote by
$H_{>\sigma}$ the open right-half plane $\{z:\Re(z)>\sigma\}$ and by
$H_{\le\sigma}$ the closed left-half plane $\{z:\Re(z)\le\sigma\}.$

\begin{corollary}\label{corollary}
There exists a Dirichlet series $S(z)=\sum_{n\ge 1}a_nn^{-z}$
absolutely convergent in the open right half-plane $H_{>1}$, which
has the following properties. For each pair of functions $\zeta_+$
and $\zeta_-$ as in the above theorems, there is a sequence of
partial sums of $S$ which converges to $\zeta$ on $H_{\le 1}$,
uniformly on admissible compact subsets of $H_{\le
1}\setminus\{1\}$; there is a sequence of partial sums of $S$ which
converges to $\zeta_+$ on $H_{\le 1}$, uniformly on admissible
compact subsets of $H_{\le 1}\setminus\{1\}$; and there is a
sequence of partial sums of $S$ which converges to $\zeta_-$ on
$H_{\le 1}$, uniformly on admissible compact subsets of $H_{\le
1}\setminus\{1\}.$ Moreover, for each $\delta>0,$ there exists such
a Dirichlet series with the further property that for $\Re(z)\ge
1+\delta,$
\begin{equation}\label{universal}
    \left|\sum_{n=1}^\infty\frac{a_n}{n^z}-\sum_{n=1}^\infty\frac{1}{n^z}\right| < \delta.
\end{equation}
\end{corollary}

\begin{proof}
As in \cite{B}, let $K_1, K_2, \cdots$ be a sequence of admissible
compacta in $H_{\le 1}\setminus\{1\}$ such that each admissible
compact subset of $H_{\le 1}\setminus\{1\}$ is contained in some
$K_k$. Set $\zeta_k=\zeta$ on $K_k$ and $\zeta_k(1)=k$. Let $S$ be a
universal Dirichlet series as in Theorem \ref{bayart}. By Theorem
\ref{bayart}, for each $k$,  there is a sequences of partial sums of
$S$:
$$
    S^{k1},S^{k2},\cdots\rightarrow\zeta_k, \,\,\,\,
    \mbox{uniformly on}\,\,K_k\cup\{1\}.
$$
A diagonal sequence of partial sums of $S$  converges  to $\zeta$ on
$H_{\le 1}$ and  uniformly on admissible compact subsets of $H_{\le
1}\setminus\{1\}.$ We do the same for $\zeta_+$ and $\zeta_-.$

By Theorem \ref{bayart}, the set of  Dirichlet series such as $S$ is
dense in the space of Dirichlet series absolutely convergent in
$H_{>1}.$ In particular, there is such a Dirichlet series $S$ which
satisfies (\ref{universal}).
\end{proof}

A special case of the preceding corollary is the following.

\begin{corollary} Let $E$ be a discrete subset of the open left half-plane
$H_{<1/2}$. For each $\delta>0,$ there exists a Dirichlet series
$S(z)=\sum_{n\ge 1}a_nn^{-z}$ absolutely convergent in the open
right half-plane $H_{>1}$ such that on the half-plane $H_{\ge
1+\delta},$
$$
    \left|\sum_{n=1}^\infty\frac{a_n}{n^z}-\sum_{n=1}^\infty\frac{1}{n^z}\right| < \delta.
$$
Moreover, a sequence of partial sums of $S$ converges to $\zeta$
uniformly on compact subsets of the critical strip
$\{z:1/2\le\Re(z)<1\}$ and to zero on $E$.
\end{corollary}

\begin{proof}
Arrange the set $E$ in a sequence $\{z_n\}$ and set
$$
    K_n = \{z:1/2\le\Re(z)\le n/(n+1)\}\cup\{z_1,z_2,\cdots,z_n\}.
$$
We define $g$ on $K_n$ by setting $g=\zeta$ on $\{z:1/2\le\Re(z)\le
n/(n+1)\}$ and $g=0$ on $\{z_1,z_2,\cdots,z_n\}.$ Then, $K_n$ is
admissible and $g$ is continuous on $K_n$ and holomorphic on the
interior.  By Theorem \ref{bayart}, there is a Dirichlet series $S$
absolutely convergent in the right half-plane $H_{>1}$ and there is
a sequence $n_1<n_2<\cdots$ such that, for each $k$,
$|S^{n_k}-g|<1/k$ on $K_k$.  By Theorem \ref{bayart}, the set of
such Dirichlet series is dense in the space of Dirichlet series
absolutely convergent in $H_{>1}.$ In particular, there is such a
Dirichlet series $S$ which satisfies (\ref{universal}).
\end{proof}

Similar results can be proved for $L$-functions.

%%%%%%%%%%%%%%%%%%%%%%%%%%%%%%%%%%%%%%%%%%%%%%%%%%%%%%%%%%%%%%%%%%%%%%%%%%%%%%%%%%%%%%%
%%%%%%%%%%%%%%%%%%%%%%%%%%%%%%%%%%%%%%%%%%%%%%%%%%%%%%%%%%%%%%%%%%%%%%%%%%%%%%%%%%%%%%%

\section{Examples of $L$-functions}

Let us describe some examples of $L$-functions in number theory
verifying the conditions we ask for.

First examples of $L$-functions are the Dirichlet $L$-functions
associated to a Dirichlet character $\chi \pmod{d}$ (see for example
\cite{Da}), that is, a function on the integers which is not
identically zero and verifies $\chi(n\cdot m)= \chi(n)\chi(m)$,
$\chi(n+d)=\chi(n)$ and $\chi(n)=0$ if $(n,d)>1$. They are defined
in the half plane $\Re(s)>1$ by the series
$$L(s,\chi):=\sum_{n=1}^{\infty} \frac{\chi(n)}{n^s}.$$
If $\chi$ is a principal character (that is, $\chi(n)=1$ or $0$ for
all $n$),  we obtain "essentially" the usual Riemann zeta-function.
In all other cases, the function $\chi$ can be extended
holomorphicly to the whole complex plane. If the character $\chi$ is
primitive and we define
\begin{equation}
Q(s):=\left(\frac d{\pi}\right)^{(s+\delta)/2}
\Gamma\left(\frac{s+\delta}2\right) , \ \mbox{ where }
\delta:=\frac{1-\chi(-1)}2,
\end{equation}
then the function $\Lambda_Q(\chi,s):=Q(s)L(\chi,s)$ verifies a
functional equation \cite{L} of the form
$$
\Lambda_Q(\chi,s):=W(\chi){\Lambda_Q}(\overline\chi,1-s),
$$
where $W(\chi)$ is an algebraic complex number with absolute value
$1$, called the root number. This equation is equivalent to the
functional equation
$$
\Lambda_Q(\chi,s):=W(\chi) \overline{\Lambda_Q}(\chi,1-s),
$$
where, for a complex function $\Lambda$, $\overline{\Lambda}(s)$ is
the complex conjugate of $\Lambda$, defined as
$\overline{\Lambda}(s):=\overline{\Lambda(\overline s)}$. To see
this equivalence, it is sufficient to notice that it holds for
positive real $s$.

So, if  $\chi$ is primitive and is  real valued (that is
$\chi(n)=\pm 1$ or $0$), and $W(\chi)=1$, then  the function
$\Lambda_Q(\chi,s):=Q(s)L(\chi,s)$, with $Q$ given by (6), satisfies
the functional equation (1). Note that Gauss proved something
equivalent to the hypothesis $W(\chi) = 1$ always being satisfied
for such real-valued characters.

More general examples are furnished by  Artin $L$-functions
associated to group representations $\rho: \mbox{Gal}(N/K)\to
\mbox{GL}_n(\C)$, where $N/K$ is a Galois extension of fields, each
of them finite extensions of $\Q$. In this case the conditions we
need are that $\rho$ be real valued (meaning its image is in
$\mbox{GL}_n(\R)$), and that the root number $W(\rho)=1$. The fact
that $\rho$ is real valued implies directly that the root number
$W(\rho)=\pm 1$, so the condition is verified automatically if
$L(\rho,1/2)\ne 0$. On the other hand, by the Fr\"{o}hlich-Queyrut
theorem \cite{FQ}, the so-called real representations (also called
real orthogonal) have root number always equal to 1 .

Easier examples of $L$-functions that {\em do} satisfy the Grand
Riemann Hypothesis are the $L$-functions associated to a (smooth
projective) curve over a finite field. In this case the functions
are always of the form $L(s)= P(q^{-s})$, where $q$ is a power of a
prime number and $P(t)$ is a polynomial with integer coefficients of
even degree $2g$ verifying the conditions, firstly, that $P(\frac
1{qt})=t^{2g}q^{-g} P(t)$ and secondly, that all (complex) roots
$\alpha$ have absolute value $| \alpha |=1/\sqrt{q}$. The first
condition implies the functional equation, taking $Q(s):=q^{-gs}$,
and the second that the zeros of $\Lambda_Q(s):=q^{-gs}P(q^{-s})$
are of the form $\log(\alpha)/\log (q)$, and so have real part equal
to $1/2$.

Some examples of other $L$-functions that are known to have a
meromorphic extension and a functional equation are the
$L$-functions associated to cuspidal modular forms $f$ (and, more
generally, to cuspidal automorphic form $\pi$). In these cases we
obtain again a functional equation of the form
$$
\Lambda_Q(f,s):=Q(s)L(f,s)=\epsilon \overline{\Lambda_Q}(f,1-s),
$$
where $Q(s)$ is a function of the form $K^s\prod_{j=1}^n
\Gamma(\lambda_j s + \mu _j)$,  $n$ is a natural number, $K$ and the
$\lambda_j$'s are positive real numbers and the $\mu_j$'s are
complex numbers with non-negative real part. So the conditions we
need are verified if $f$ is a so-called real modular form (or with
real coefficients) and $\epsilon=1$. This last condition is again
automatic if $L(f,1/2)\ne 0$, since $\epsilon$ must be $\pm 1$ if
the form is real.

This last case includes some $L$-functions of more geometric origin:
the $L$-functions associated to elliptic curves defined over $\Q$.
After the work of Andrew Wiles and others (see \cite{W} and
\cite{BCDT}) it is known that the $L$-function of such a curve is
equal to the $L$-function of a (real) modular form (of weight 2).
From the work of Victor A. Kolyvagin \cite{Ko} it is also known that
the condition $L(f,1/2)\ne 0$ is equivalent to the condition the
curve have only finitely many points over $\Q$. On the other hand,
the condition that the root number be equal to 1 should be
equivalent, according to the famous Birch and Swinnerton-Dyer
conjecture (see for example the Bourbaki talk \cite{T}), to the rank
of the group of rational points of the curve being even. Note that
we are considering a modified version of the usual $L$-function,
which, in this case, has a functional equation relating $s$ with
$2-s$.

Note that all the cases presented are Dirichlet series having an
Euler product. Recall that a complex function $f$ has an Euler
product if for $\Re(s)>1$ it can be expressed as a Dirichlet series
and an infinite product varying on the primes $p$
$$f(s)=\sum_{n=1}^{\infty}\lambda(n)n^{-s}=\prod_{p} f_p(s)$$
such that
$$f_p(s)=\prod_{i=1}^{d_p} (1-\alpha_i(p)p^{-s})^{-1}$$
where $\alpha_i(p)$ are complex numbers such that $|\alpha_i(p)|<p$,
$\lambda(n)$ are also complex numbers, and $\lambda(1)=1$ and
$d_p\ge 1$ is a natural number, which in the classical cases is
independent of $p$ (see for example \cite{IK}, chapter 5). It would
be very interesting to show that we can approximate by functions
having some of these properties.

\begin{remark} The results in this
paper can be modified to cover some of the other cases of functional
equations for $L$-functions as explained in the examples above. For
example, one can consider $L$-functions with functional equation of
the form
$$
\Lambda_Q(f,s):=Q(s)f(s)=-\Lambda_Q(f,1-s)
$$
by considering functions antisymmetric with respect to the point
$s=1/2$, i.e  $f(1-s)=-f(s)$, for all $s\in\C$. Such functions must
have a zero for $s=1/2$.
\end{remark}

We thank Andrew Granville and J\"orn Steuding for reading our
manuscript and making helpful suggestions and especially Markus
Niess who, after reading the next-to-last version in detail, pointed
out errors and suggested simplifications . We also thank Javad
Mashreghi. Corollary \ref{product} is in response to a question he
posed, when one of us lectured on this topic.

\end{document}